\documentclass[review,12pt]{elsarticle}

\usepackage[margin=1in]{geometry}
\usepackage{amsmath,mathrsfs}
\usepackage{amssymb,amsthm}
\usepackage{graphicx}
\usepackage{paralist}
\usepackage{mathdots}

\newcommand{\real}{\mathbb{R}}

\newtheorem{theorem}{Theorem}[section]
\newtheorem{corollary}{Corollary}[section]
\newtheorem{proposition}{Proposition}[section]
\newtheorem{lemma}{Lemma}[section]
\theoremstyle{definition}

\newtheorem{example}{Example}[section]
\newtheorem{remark}{Remark}[section]
\newtheorem{conjecture}{Conjecture}[section]

\begin{document}

\begin{frontmatter}

\title{The role of the anti-regular graph in the\\ spectral analysis of threshold graphs}

\author{Cesar O. Aguilar}
\address{Department of Mathematics, State University of New York, Geneseo}

\author{Matthew Ficarra}
\address{Department of Mathematics, State University of New York, Geneseo}

\author{Natalie Schurman}
\address{Department of Mathematics, State University of New York, Geneseo}

\author{Brittany Sullivan}
\address{Department of Mathematics, State University of New York, Geneseo}

\begin{abstract}
The purpose of this paper is to highlight the role played by the anti-regular graph within the class of threshold graphs.  Using the fact that every threshold graph contains a maximal anti-regular graph, we show that some known results, and new ones, on the spectral properties of threshold graphs can be deduced from (i) the known results on the eigenvalues of anti-regular graphs, (ii) the subgraph structure of threshold graphs, and (iii) eigenvalue interlacing.  In particular, we prove a strengthened version of the recently proved fact that no threshold graph contains an eigenvalue in the interval $\Omega = [\frac{-1-\sqrt{2}}{2},\frac{-1+\sqrt{2}}{2}]$, except possibly the trivial eigenvalues $-1$ and/or $0$, determine the inertia of a threshold graph, and give partial results on a conjecture regarding the optimality of the non-trivial eigenvalues of an anti-regular graph within the class of threshold graphs.  
\end{abstract}


\end{frontmatter}



\baselineskip 1.5em

\section{Introduction}
A simple graph $G=(V,E)$ is a \textit{threshold} graph if there exists a function $w:V\rightarrow [0,\infty)$ and a real number $t\geq 0$ called the \textit{threshold} such that for every $X\subset V$, $X$ is an independent set if and only if $\sum_{v\in X} w(v) \leq t$.  Threshold graphs were independently introduced in \cite{VC-PH:77} and \cite{PH-YZ:77}; for a comprehensive survey of threshold graphs see \cite{NM-UP:95}.  Threshold graphs have applications in resource allocation problems where the weight $w(v)$ is the amount of resources used by vertex $v$ and thus $X$ is an admissible subset of vertices if the total amount of resources required by $X$ is no more than the allowable threshold $t$.  

In this paper, we are interested in the eigenvalues of the $(0,1)$-adjacency matrix $A(G)$ of a threshold graph $G$.  To the best of the authors' knowledge, the first study on the spectral properties of threshold graphs was focused on a specific threshold graph called the \textit{anti-regular} graph \cite{EM:09}.  In \cite{EM:09}, several recurrence relations were obtained for the characteristic polynomial of the unique $n$-vertex connected anti-regular graph $A_n$, and moreover it was shown that $A_n$ has simple eigenvalues.  Specifically,  $A_n$ has $\lfloor\frac{n}{2}\rfloor$ negative and $\lfloor\frac{n}{2}\rfloor$ positive eigenvalues; when $n$ is odd $A_n$ has a zero eigenvalue and when $n$ is even $A_n$ has $-1$ as an eigenvalue.  Subsequently in \cite{IS-SF:11}, it was proved that the eigenvalues of $A_n$ other than $-1$ or $0$ are \textit{main eigenvalues}; recall that $\lambda$ is a main eigenvalue of $G$ if the eigenspace associated to $\lambda$ is not orthogonal to the all ones vector (see \cite{DC-PR-SS:10}).  In \cite{RBT:13}, the inertia of a general threshold graph was computed from the binary string uniquely associated to a threshold graph and moreover the inverse of the adjacency matrix of some threshold graphs were computed.  In \cite{DJ-VT-FT:13}, an algorithm is presented that constructs a diagonal matrix congruent to $A(G)+xI$; using the algorithm one can determine the number of eigenvalues of $G$ in any given interval.  Moreover, in \cite{DJ-VT-FT:13} the authors determine the threshold graph with smallest negative eigenvalue and show that all eigenvalues of a threshold graph are simple except possibly $-1$ and/or $0$.  In \cite{DJ-VT-FT:14}, the authors present an $O(n^2)$ algorithm for computing the characteristic polynomial of an $n$-vertex threshold graph and an improved algorithm running in almost linear time was constructed in \cite{MF:17}.  In \cite{DJ-VT-FT:15}, it is proved that no threshold graph has an eigenvalue in the interval $(-1,0)$ and a study of noncospectral equienergetic threshold graphs was undertaken.  In \cite{AB-RM:17}, the authors investigate the normalized adjacency eigenvalues and energy of threshold graphs and they obtain results that parallel the known results for the adjacency eigenvalues.  In \cite{CA-EP-JL-BS:18}, a nearly complete characterization of the eigenvalues of anti-regular graphs is given.  Specifically, it is proved that no anti-regular graph has an eigenvalue in the interval $\Omega = [\frac{-1-\sqrt{2}}{2},\frac{-1+\sqrt{2}}{2}]$ other than $-1$ or $0$, and moreover, the eigenvalues of $A_n$ come in negative-positive pairs in the sense that that the vertical line $x=-\frac{1}{2}$ is an approximate line of symmetry of the paired eigenvalues.  Furthermore, the set of all eigenvalues of all anti-regular graphs $A_n$ is dense in $(-\infty, \frac{-1-\sqrt{2}}{2}] \cup \{-1\} \cup [\frac{-1+\sqrt{2}}{2},\infty)$ when $n$ is even and dense in $(-\infty, \frac{-1-\sqrt{2}}{2}] \cup \{0\} \cup [\frac{-1+\sqrt{2}}{2},\infty)$ when $n$ is odd.  It was conjectured in \cite{CA-EP-JL-BS:18} that $\Omega$ is also an eigenvalue-free interval (except $-1$ and/or $0$) for \textit{all} threshold graphs.  It was also conjectured that among all threshold graphs on $n$ vertices, $A_n$ has the smallest positive eigenvalue and the largest eigenvalue less than $-1$.  In \cite{ZL-JW-QH:19}, the authors use the quotient graph associated to the degree partition (which is an \textit{equitable partition} \cite{CG-GR:01}) of a threshold graph to derive the known results on the inertia of a threshold graph and determine which threshold graphs have distinct eigenvalues.  Finally, in \cite{JL-OM-FT:19} the authors derive an explicit expression for the characteristic polynomial of a threshold graph and use it to find the determinant of $A(G)$ and prove that no two non-isomorphic graphs are cospectral.  Recently in \cite{EG:19}, a proof of the $\Omega$-conjecture was given using eigenvalue interlacing.

In this paper, we provide a more refined result on the conjecture in \cite{CA-EP-JL-BS:18} regarding the $\Omega$ interval, give partial results for the second conjecture and identify the critical cases where a more refined method is needed.    Perhaps more importantly, in this paper we demonstrate the distinguished role played by the anti-regular graph within the class of threshold graphs.  Specifically, we exploit the observation that every threshold graph contains a maximal anti-regular graph as an induced subgraph and conversely every threshold graph is an induced subgraph of a minimal anti-regular graph.  Using this observation we are able to give a new straightforward proof for the inertia of any threshold graph and obtain our main results regarding the conjectures in \cite{CA-EP-JL-BS:18}.  We also provide estimates for the maximum and minimum eigenvalue of a general threshold graph using easy to analyze induced threshold subgraphs contained in any threshold graph.

\section{Preliminaries}
Let $G=(V,E)$ be a simple graph with $(0,1)$-adjacency matrix $A=A(G)$.  Whenever we refer to the eigenvalues, eigenvectors, inertia, etc.\ of $G$ we mean those of $A$.  If $n$ is the order of $G$, we denote the eigenvalues of $G$ by $\lambda_1(G) \leq \lambda_2(G) \leq \cdots \leq \lambda_n(G)$ and we let $\mu^-(G)$ denote the largest eigenvalue of $G$ less than $-1$ (when such an eigenvalue exists) and $\mu^+(G)$ the smallest positive eigenvalue of $G$.  The \textit{inertia} of $A$ is the triple $i(A) = (i_{-}(A), i_0(A), i_+(A))$ where $i_-(A)$ is the number of negative, $i_+(A)$ is the number of positive, and $i_0(A)$ is the number of zero eigenvalues of $A$.  The following well-known eigenvalue interlacing theorem will be used throughout the paper.

\begin{theorem}[Eigenvalue Interlacing]\label{}
Let $G$ be an $n$-vertex graph and let $H$ be an $m$-vertex induced subgraph of $G$.  Then for $i\in \{1,\ldots,m\}$ we have
\[
\lambda_i(G) \leq \lambda_i(H) \leq \lambda_{n-m+i}(G).
\]
\end{theorem}

Threshold graphs have several equivalent characterizations \cite{NM-UP:95}; the most illuminating is a recursive process, using the union and join graph operations, that can be encoded with a binary string.  Given a binary string $b=b_1b_2\cdots b_n$ with $b_1=0$, we let $G_1=(\{v_1\}, \emptyset)$ and then recursively define for $j=2,\ldots,n$ a graph $G_j$ obtained from $G_{j-1}$ by adding a new vertex $v_j$ and making $v_j$ a dominating vertex if $b_j=1$, or leaving $v_j$ as an isolated vertex if $b_j=0$.  After the $n$th step the resulting graph $G = G(b)$ is a threshold graph;  $G$ is clearly connected if and only if $b_n=1$.  We refer to the resulting labelled vertex set $V(G)=\{v_1,v_2,\ldots,v_n\}$ as the \textit{canonical labeling} of $G$.  

Let $G$ be a connected threshold graph with binary string $b=b_1b_2\cdots b_n$ and canonically labelled vertex set $V(G)$.  The string $b$ can written as $b=0^{s_1}1^{t_1}\ldots 0^{s_k}1^{t_k}$ where $0^{s_i}$ is short-hand for $s_i\geq 1$ consecutive zeros and $1^{t_i}$ is short-hand for $t_i\geq 1$ consecutive ones.  Since $n=\sum_{i=1}^k(s_i+t_i)$, it holds that $1\leq k \leq \lfloor\frac{n}{2}\rfloor$.  We can partition the vertex set as $V(G)=U_1 \cup V_1 \cup \cdots \cup U_k\cup V_k$ where the set $U_i$ consists of the $i$th group of consecutive isolated vertices in the construction of $G$ and thus $|U_i|=s_i$, and similarly, $V_i$ consists of the $i$th group of dominating vertices in the construction of $G$ and thus $|V_i|=t_i$.  If $s_1\geq 2$ then $\{U_1,V_1,\ldots,U_k,V_k\}$ is the degree partition of $G$ while if $s_1=1$ then the degree partition is $\{U_1\cup V_1, U_2, V_2,\ldots, U_k, V_k\}$.  In any case, each subset $U_i$ is an independent set and each subset $V_i$ is a clique.  Figure~\ref{fig:threshold-structure} illustrates the degree partition of a threshold graph; a line between $U_i$ and $V_j$ indicates that all vertices in $U_i$ are adjacent to all vertices in $V_j$, and the dashed rectangle indicates that $V_1\cup\cdots\cup V_k$ is a clique.  
\begin{figure}
\centering
\includegraphics[width=120mm,keepaspectratio,trim=0 0.5in 0 0.5in,clip]{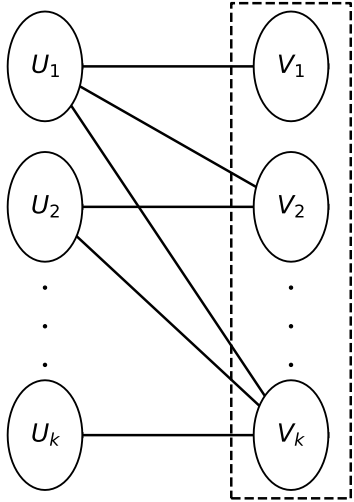}
\caption{Global structure of a threshold graph with binary string $b=0^{s_1}1^{t_1}\cdots 0^{s_k}1^{t_k}$; each vertex in $U_i$ is adjacent to $V_i \cup \cdots \cup V_k$, and $V_1\cup\cdots\cup V_k$ is a clique and $U_1\cup\cdots\cup U_k$ is an independent set.}\label{fig:threshold-structure}
\end{figure}  

The connected \textit{anti-regular} graph on $n$ vertices, denoted by $A_n$, is the unique connected graph whose degree sequence contains $(n-1)$ distinct entries \cite{MB-GC:67}.  The graph $A_n$ is a threshold graph with binary string $b=0101\cdots 01$ when $n$ is even and $b=00101\cdots 01$ when $n$ is odd.  It was proved in \cite{EM:09} (see also \cite{CA-EP-JL-BS:18}) that $A_n$ has simple eigenvalues and moreover has inertia $i(A_{2k}) = (k, 0, k)$ if $n=2k$ is even and $i(A_{2k+1}) = (k,1,k)$ if $n=2k+1$ is odd, and therefore $\lambda_{k+1}(A_{2k+1}) = 0$.  Moreover, it is known that $A_{2k+1}$ does not contain $-1$ as an eigenvalue and it is easy to prove that $A_{2k}$ has eigenvalue $\lambda_{k}(A_{2k}) = -1$.  Therefore,
\begin{align*}
\mu^-(A_{2k}) &= \lambda_{k-1}(A_{2k})\\
\mu^+(A_{2k}) &= \lambda_{k+1}(A_{2k}),
\end{align*}
and
\begin{align*}
\mu^-(A_{2k+1}) &= \lambda_k(A_{2k+1})\\ 
\mu^+(A_{2k+1}) &= \lambda_{k+2}(A_{2k+1}).
\end{align*}
The following was proved in \cite{CA-EP-JL-BS:18}.
\begin{lemma}[Parity Principle]\label{thm:parity-principle}
The sequences $\{\mu^-(A_{2k})\}_{k=2}^\infty$ and $\{\mu^-(A_{2k+1})\}_{k=1}^\infty$ are strictly increasing and both converge to $\frac{-1-\sqrt{2}}{2}$.  Similarly, $\{\mu^+(A_{2k})\}_{k=2}^\infty$ and $\{\mu^+(A_{2k+1})\}_{k=1}^\infty$ are strictly decreasing sequences and both converge to $\frac{-1+\sqrt{2}}{2}$.
\end{lemma}
It is important to note that Lemma~\ref{thm:parity-principle} does not say that $\{\mu^-(A_n)\}_{n=3}^\infty$ is strictly increasing nor that $\{\mu^+(A_n)\}_{n=3}^\infty$ is strictly decreasing.  In fact, the opposite is true depending on whether $n$ is even or odd.  To see this, we first note that $A_n$ is an induced subgraph of $A_{n+1}$ (see Section~\ref{sec:subgraphs}).  Hence, by eigenvalue interlacing it holds that
\[
\mu^-(A_{2k}) = \lambda_{k-1}(A_{2k}) \leq \lambda_{k}(A_{2k+1}) = \mu^-(A_{2k+1})
\]
while on the other hand
\[
\mu^-(A_{2k+2}) = \lambda_k(A_{2k+2})\leq \lambda_{k}(A_{2k+1}) = \mu^-(A_{2k+1}).
\]
Similarly, 
\[
\mu^+(A_{2k}) = \lambda_{k+1}(A_{2k}) \leq \lambda_{k+2}(A_{2k+1}) = \mu^+(A_{2k+1})
\]
while on the other hand
\[
\mu^+(A_{2k+2}) = \lambda_{k+2}(A_{2k+2}) \leq \lambda_{k+2}(A_{2k+1}) = \mu^+(A_{2k+1}).
\]
We summarize with the following.
\begin{proposition}\label{prop:odd-even}
If $n\geq 3$ is odd, then 
\begin{align*}
\mu^-(A_{n+1}) &\leq \mu^-(A_n)\\
\mu^+(A_{n+1}) &\leq \mu^+(A_n).
\end{align*}
\end{proposition}
\noindent We note that Proposition~\ref{prop:odd-even} implies that if $n\geq 4$ is even, then
\begin{align*}
\mu^-(A_n) &\leq \mu^-(A_{n+1})\\
\mu^+(A_n) &\leq \mu^+(A_{n+1}).  
\end{align*}

\section{Subgraphs in threshold graphs}\label{sec:subgraphs}
It is straightforward to show that \textit{any} induced subgraph of a threshold graph is again a threshold graph.  In fact, suppose that $b'=b_1'b_2'\cdots b_m'$ is a \textit{substring} of the string $b=b_1b_2\cdots b_n$ with $b_1=b_1'=0$, that is, there exist positive integers $n_2 < \cdots < n_m$ such that $b'_j = b_{n_j}$ for $j=2,\ldots,m$, with $n_2\geq 2$.  Then the threshold graph $G(b')$ is isomorphic to the subgraph of $G(b)$ induced by the vertices $\{v_{1},v_{n_2},\ldots,v_{n_m}\}$, where as usual $G(b)$ has a canonically labelled vertex set.  We summarize this observation with the following.
\begin{proposition}\label{thm:subgraph-threshold}
Let $b'=b_1'b_2'\cdots b_m'$ be a substring of $b=b_1b_2\cdots b_n$ with $b_1=b_1'=0$.  Then $G(b')$ is isomorphic to an induced subgraph of $G(b)$.  Conversely, every induced subgraph of $G(b)$ is of the form $G(b')$.  
\end{proposition}

In \cite{RM:03}, R. Merris proved that anti-regular graphs are \textit{universal} for trees, that is, every tree on $n$ vertices is isomorphic to a subgraph of $A_n$.  We show that a similar universality property of the anti-regular graph holds for the class of threshold graphs.
\begin{theorem}[Smallest Anti-regular Supergraph]\label{thm:antiregular-supergraph}
Every connected threshold graph on $n\geq 2$ vertices is isomorphic to an induced subgraph of the anti-regular graph $A_{2n-2}$.  In fact, let $G$ be a connected threshold graph with binary string $b=0^{s_1}1^{t_1}\cdots 0^{s_k}1^{t_k}$ and with $n=\sum_{i=1}^k (s_i+t_i)$ vertices.  Let $N=2(n-k)$ if $s_1=1$ and let $N=2(n-k)-1$ if $s_1\geq 2$.  Then $G$ is an induced subgraph of the anti-regular graph $A_N$.  Moreover, $A_N$ is the smallest anti-regular graph containing $G$ as an induced subgraph.
\end{theorem}
\begin{proof}
The shortest alternating string $\overline{b}=\overline{b}_1\overline{b}_2\cdots\overline{b}_N=0101\cdots 01$ that contains $b=0^{s_1}1^{t_1}\cdots 0^{s_k}1^{t_k}$ as a substring can be obtained by inserting $(t_i-1)$ zeros in between the consecutive $t_i$ ones and inserting $(s_i-1)$ ones in between the consecutive $s_i$ zeros appearing in $b$, for $i=1,\ldots,k$.  Therefore, $$N= n + \sum_{i=1}^k (s_i-1) + \sum_{i=1}^k (t_i-1) = 2n - 2k.$$  Since the range of $k$ is $1\leq k \leq \lfloor\tfrac{n}{2}\rfloor$, we have $N\leq 2n-2$.  Hence, $G$ is an induced subgraph of $A_N$ and thus also of $A_{2n-2}$.  If $s_1=1$, then we obtain exactly $N=2(n-k)$ but if $s_1\geq 2$, then we may embed $G$ in the smaller anti-regular graph $A_{2(n-k)-1}$.  
\end{proof}
Conversely, we may be interested in the largest anti-regular graph contained in a threshold graph.

\begin{theorem}[Largest Anti-regular Subgraph]\label{thm:antiregular-subgraph}
Let $G$ be a connected threshold graph with binary string $b=0^{s_1}1^{t_1}\cdots 0^{s_k}1^{t_k}$.  Let $m=2k$ if $s_1=1$ and let $m=2k+1$ if $s_1\geq 2$.  Then $A_m$ is an induced subgraph of $G$.  In either case, $A_m$ is the largest anti-regular graph contained in $G$ as an induced subgraph.
\end{theorem}
\begin{proof}
If $s_1=1$, then $\tilde{b}=0101\ldots 01\in\{0,1\}^{2k}$ is the longest alternating substring of $b$ and thus $A_{2k}$ is isomorphic to an induced subgraph of $G$.  If $s_1 \geq 2$, then $A_{2k}$ is also an induced subgraph of $G$, however the longer binary string $\tilde{b}=00101\cdots 01\in\{0,1\}^{2k+1}$ is also a substring of $b$ and thus $A_{2k+1}$ is isomorphic to an induced subgraph of $G$.  
\end{proof}
We note that part of Theorem~\ref{thm:antiregular-subgraph} is stated as Corollary 4.4 in \cite{IS-SF:11}.
\begin{remark}
The anti-regular graph $A_m$ in Theorem~\ref{thm:antiregular-subgraph} is the underlying graph of the quotient graph of $G$ associated to the degree partition of $G$.
\end{remark}

\section{Applications in the spectral analysis of threshold graphs}\label{sec:main-results}
In this section we show how the Parity Principle and Theorems~\ref{thm:antiregular-supergraph}-\ref{thm:antiregular-subgraph} can be used in the spectral analysis of general threshold graphs.  For a matrix $A$ we denote the algebraic multiplicity of an eigenvalue $\lambda$ of $A$ by $m_\lambda(A)$.  

For any vertex $u$ in $G$, let $N(u)$ denote the vertices adjacent to $u$.  We say that $v_i$ and $v_j$ are \textit{duplicate} vertices if $N(v_i) = N(v_j)$ and \textit{co-duplicate} vertices if $v_i$ and $v_j$ are adjacent and $N(v_i)\backslash\{v_j\} = N(v_j)\backslash\{v_i\}$.  It is straightforward to show that if $v_i$ and $v_j$ are duplicate or co-duplicate vertices, then $\lambda=0$ or $\lambda=-1$, respectively, is an eigenvalue of $G$ with eigenvector $x\in\real^n$ such that $x_i = - x_j$ and all other entries of $x$ are zero.  It follows that if $X\subset V(G)$ is a subset of mutually duplicate or co-duplicate vertices, then $m_0(G)$ or $m_{-1}(G)$, respectively, is at least $|X|-1$.  Thus, given a connected threshold graph $G$ with binary string $b=0^{s_1}1^{t_1}\cdots 0^{s_k}1^{t_k}$, it holds that $m_{-1}(G) \geq \sum_{i=1}^k (t_i-1)$ and $m_0(G)\geq \sum_{i=1}^k (s_i-1)$ if $s_1\geq 2$, while $m_{-1}(G) \geq t_1 + \sum_{i=2}^k (t_i-1)$ and $m_{0}(G) \geq \sum_{i=2}^k (s_i-1)$ if $s_1=1$.  For these reasons, for a threshold graph $G$ with eigenvalue $\lambda$, we say that $\lambda$ is a \textit{non-trivial} eigenvalue if $\lambda \notin \{-1,0\}$.  

\begin{theorem}(Anti-regular Interlacing)\label{thm:antiregular-interlacing}
Let $G$ be a connected threshold graph with binary string $b=0^{s_1}1^{t_1}\cdots 0^{s_k}1^{t_k}$.  Let $s=\sum_{i=1}^k s_i$ and let $t=\sum_{i=1}^k t_i$, and let $n=s+t=|G|$.
\begin{enumerate}[(i)]
\item If $s_1\geq 2$, then
\[
\lambda_i(G) \leq \lambda_i(A_{2k+1}) < -1, \; \textup{ for } i=1,\ldots, k
\]
and
\[
0<\lambda_{k+1+i}(A_{2k+1}) \leq \lambda_{n-k+i}(G), \;\textup{ for } i=1,\ldots, k.
\]
Consequently, $m_{-1}(G) = t-k$ and $m_{0}(G) = s-k$, and $G$ has $k$ non-trivial negative and $k$ non-trivial positive eigenvalues.

\item If $s_1=1$, then
\[
\lambda_i(G) \leq \lambda_i(A_{2k})<-1, \; \textup{ for } i=1,\ldots, k-1
\]
and
\[
0<\lambda_{k+i}(A_{2k}) \leq \lambda_{n-k+i}(G),\; \textup{ for } i=1,\ldots, k.
\]
Consequently, $m_{-1}(G) = t-k+1$ and $m_0(G)=s-k$, and $G$ has $(k-1)$ non-trivial negative and $k$ non-trivial positive eigenvalues.
\end{enumerate}
In either case, $G$ has inertia $i(G)=(t, s-k, k)$.
\end{theorem}
\begin{proof}
The proof is a consequence of Theorem~\ref{thm:antiregular-subgraph}, the eigenvalue interlacing theorem, and the remarks preceding the theorem statement.  For instance, if $s_1\geq 2$, then Theorem~\ref{thm:antiregular-subgraph} implies that $A_{2k+1}$ is an induced subgraph of $G$.  Then the inequalities in (i) hold by the interlacing theorem.  Now, since $m_{-1}(G) \geq t-k$ and $m_{0}(G)\geq s-k$, then the inequalities in (i) imply that in fact $m_{-1}(G) = t-k$ and $m_0(G) = s-k$, and thus $G$ has $k$ non-trivial negative and $k$ non-trivial positive eigenvalues.  The case $s_1=1$ is similar and is omitted.
\end{proof}
Recall that $\mu^-(G)$ denotes the largest eigenvalue of $G$ less than $-1$ and $\mu^+(G)$ denotes the smallest positive eigenvalue of $G$.  A direct consequence of Theorem~\ref{thm:antiregular-interlacing} is the following.
\begin{corollary}\label{cor:omega-refined}
Let $G$ be a connected threshold graph with binary string $b=0^{s_1}1^{t_1}\cdots 0^{s_k}1^{t_k}$.  
\begin{enumerate}[(i)]
\item If $s_1\geq 2$, then $G$ does not contain non-trivial eigenvalues in the interval $$[\mu^-(A_{2k+1}), \mu^+(A_{2k+1})].$$
\item If $s_1=1$, then $G$ does not contain non-trivial eigenvalues in the interval $$[\mu^-(A_{2k}), \mu^+(A_{2k})].$$
\end{enumerate}
\end{corollary}
In \cite{CA-EP-JL-BS:18} it is proved that $\Omega = [\frac{-1-\sqrt{2}}{2},\frac{-1+\sqrt{2}}{2}]$ does not contain non-trivial eigenvalues of any anti-regular graph $A_n$ for $n\geq 2$, that is, $\Omega \subsetneq [\mu^-(A_n),\mu^+(A_n)]$.  We may therefore conclude the following. 
\begin{corollary}\label{cor:free-interval}
The interval $\Omega = [\frac{-1-\sqrt{2}}{2},\frac{-1+\sqrt{2}}{2}]$ does not contain any non-trivial eigenvalues of any threshold graph.
\end{corollary}
Using a similar eigenvalue interlacing technique and an induction argument, Corollary 4.2 was first proved by E. Ghorbani \cite{EG:19}.  It is clear, however, that Corollary~\ref{cor:omega-refined} gives a more refined estimate for an eigenvalue-free interval for any given threshold graph $G$ in terms of the largest anti-regular subgraph of $G$.

\begin{remark}
Corollary~\ref{cor:free-interval} can also be proved using Theorem~\ref{thm:antiregular-supergraph} and the corresponding analog of Theorem~\ref{thm:antiregular-interlacing}.
\end{remark}

The following conjecture was made in \cite{CA-EP-JL-BS:18}.
\begin{conjecture}\label{conj:An-optimal}
For each $n$, the anti-regular graph $A_n$ has the smallest positive eigenvalue and has the largest negative eigenvalue less than $-1$ among all threshold graphs on $n$ vertices.
\end{conjecture}

As a final application of Theorem~\ref{thm:antiregular-subgraph}, we are able to prove that Conjecture~\ref{conj:An-optimal} is true for all threshold graphs on $n$ vertices except for $n-2$ critical cases where the interlacing method fails; roughly speaking, the critical graphs are \textit{almost} anti-regular.
\begin{theorem}\label{thm:even-opt}
Assume that $n\geq 2$ is even.  
\begin{enumerate}[(i)]
\item Then $\mu^+(A_n) \leq \mu^+(G)$ for every threshold graph $G$ on $n$ vertices.
\item Then $\mu^-(G)\leq\mu^-(A_n)$ for every threshold graph $G$ on $n$ vertices with binary string $b=0^{s_1}1^{t_1}\cdots 0^{s_k}1^{t_k}$ with $s_1=1$ .
\item Then $\mu^-(G)\leq \mu^-(A_n)$ for every threshold graph $G$ on $n$ vertices with binary string $b=0^{s_1}1^{t_1}\cdots 0^{s_k}1^{t_k}$ with $s_1\geq 2$ and $2k+2<n$.
\end{enumerate}
\end{theorem}
\begin{proof}
(i) Assume that $G$ has binary string $b=0^{s_1}1^{t_1}\cdots 0^{s_k}1^{t_k}$ and thus $2k\leq n$.  Then Theorem~\ref{thm:antiregular-interlacing} implies that $\mu^+(G) = \lambda_{n-k+1}(G)$.  Now $A_{2k}$ is an induced subgraph of $G$ and thus by interlacing we have 
\[
\mu^+(A_{2k}) = \lambda_{k+1}(A_{2k}) \leq \lambda_{n-2k+(k+1)}(G) = \mu^+(G).
\]
Since $n\geq 2k$, then by the Parity Principle (Lemma~\ref{thm:parity-principle}) we have $\mu^+(A_n)\leq \mu^+(A_{2k})$ and thus $\mu^+(A_n) \leq \mu^+(G)$ as claimed.

(ii) Now suppose that $s_1=1$, and therefore $\mu^-(G) = \lambda_{k-1}(G)$.  Since $A_{2k}$ is a subgraph of $G$, then by interlacing we have
\[
\mu^-(G) = \lambda_{k-1}(G) \leq \lambda_{k-1}(A_{2k}) = \mu^-(A_{2k}) \leq \mu^-(A_n)
\]
where the last inequality holds by the Parity Principle since $n$ is even.

(iii) If $s_1\geq 2$, then $\mu^-(G) = \lambda_k(G)$.  The proof is by strong induction.  The case $n=2$ is trivial.  Assume that the claim holds for all threshold graphs with less than $n$ vertices.  Since $2k+2<n$, there exists a threshold subgraph $\tilde{G}$ of $G$ with binary string $\tilde{b}=0^{\tilde{s}_1}1^{\tilde{t}_1}\cdots 0^{\tilde{s}_k}1^{\tilde{t}_k}$ with $\tilde{s}_1 \geq 2$ and $|\tilde{G}| = 2k+2$.  Thus $\mu^-(\tilde{G}) = \lambda_k(\tilde{G})$ and by induction $\mu^-(\tilde{G}) \leq \mu^-(A_{2k+2})$.  Then by interlacing and the induction hypothesis we have
\[
\mu^-(G) = \lambda_k(G) \leq \lambda_k(\tilde{G}) \leq \mu^-(A_{2k+2}).
\]
Since $n$ is even, then the Parity Principle implies that $\mu^-(A_{2k+2}) \leq \mu^-(A_n)$ and thus $\mu^-(G) \leq \mu^-(A_n)$.
\end{proof}
The threshold graphs for which the method of proof in Theorem~\ref{thm:even-opt}(iii) fails to apply have binary string $b=0^{s_1}1^{t_1}\cdots 0^{s_k}1^{t_k}$ such that $s_1\geq 2$ and $n=2k+2$.  It follows that either $s_1=2$ and exactly one of $s_2,\ldots,s_k, t_1, \ldots,t_k$ is also equal to two and all others are one, or $s_1=3$ and all other $s_i=t_i=1$.  Hence, there are only $2k = n-2$ threshold graphs not covered by Theorem~\ref{thm:even-opt}(iii).  Note that these graphs are almost anti-regular and contain $A_{2k+1}$ as an induced subgraph.  For instance, if $n=8$ the graphs are
\begin{gather*}
0^21^20101, 0^210^2101, 0^2101^201,  0^21010^21, 0^210101^2, 0^310101.
\end{gather*}
Since $\mu^{-}(A_{2k+2}) \leq \mu^-(A_{2k+1})$ (Proposition~\ref{prop:odd-even}) the method of proof in Theorem~\ref{thm:even-opt}(iii) will not yield $\mu^-(G)\leq \mu^-(A_n)$ for these critical graphs $G$.  We now treat the case $n$ odd.

\begin{theorem}\label{thm:odd-opt}
Assume $n\geq 3$ is odd.  
\begin{enumerate}[(i)]
\item Then $\mu^-(G) \leq \mu^-(A_n)$ for all threshold graphs $G$ on $n$ vertices.
\item Then $\mu^+(G) \leq \mu^+(A_n)$ for all threshold graphs $G$ on $n$ vertices with binary string $b=0^{s_1}1^{t_1}\cdots 0^{s_k}1^{t_k}$ with $s_1\geq 2$ .
\item Then $\mu^+(G) \leq \mu^+(A_n)$ for all threshold graphs $G$ on $n$ vertices with binary string $b=0^{s_1}1^{t_1}\cdots 0^{s_k}1^{t_k}$ with $s_1 = 1$ and $2k+1<n$.
\end{enumerate}
\end{theorem}
\begin{proof}
(i) Let $G$ have binary string $b=0^{s_1}1^{t_1}\cdots 0^{s_k}1^{t_k}$.  Suppose that $s_1=1$ and thus $\mu^-(G) = \lambda_{k-1}(G)$.  Since $A_{2k}$ is an induced subgraph of $G$ then by interlacing
\[
\lambda_{k-1}(G) \leq \lambda_{k-1}(A_{2k}) = \mu^-(A_{2k}) \leq \mu^-(A_n)
\]
where the last inequality follows since from even to odd the negative eigenvalue increases (Proposition~\ref{prop:odd-even}).  If on the other hand $s_2\geq 2$, then $\mu^-(G) = \lambda_k(G)$.  In this case, $A_{2k+1}$ is an induced subgraph of $G$ and $\mu^-(A_{2k+1}) = \lambda_k(A_{2k+1})$.  Then
\[
\lambda_k(G) \leq \lambda_k(A_{2k+1}) \leq \mu^-(A_n)
\]
since $n$ is odd and thus $\mu^-(G) \leq \mu^-(A_n)$.

(ii) Suppose that $s_1 \geq 2$.  Then $A_{2k+1}$ is a subgraph of $G$.  Now $\mu^+(G) = \lambda_{n-k+1}(G)$ and $\mu^+(A_{2k+1}) = \lambda_{k+2}(A_{2k+1})$.  Therefore by interlacing
\[
\mu^+(A_n) \leq \mu^+(A_{2k+1}) = \lambda_{k+2}(A_{2k+1}) \leq \lambda_{n-(2k+1)+k+2}(G) = \lambda_{n-k+1}(G) = \mu^+(G)
\]
where the first inequality follows since $n$ is odd.

(iii) Suppose that $s_1=1$.  The proof is by strong induction.  The case $n=3$ is trivial.  Hence, assume that the claim holds for all threshold graphs with an odd number of vertices less than $n$.  Let $G$ be an arbitrary threshold graph on $n$ vertices with string $b=0^{s_1}1^{t_1}\cdots 0^{s_k}1^{t_k}$ with $s_1 = 1$ and $2k+1<n$.  Let $\tilde{G}$ be any threshold subgraph of $G$ of order  $2k+1$ and with binary string $\tilde{b}=0^{\tilde{s}_1}1^{\tilde{t}_1}\cdots 0^{\tilde{s}_k}1^{\tilde{t}_k}$ with $\tilde{s}_1=1$.  Then $\mu^+(\tilde{G}) = \lambda_{k+2}(\tilde{G})$ and by interlacing
\[
\mu^+(\tilde{G})=\lambda_{k+2}(\tilde{G}) \leq \lambda_{n-(2k+1)+k+2}(G) = \lambda_{n-k+1}(G) = \mu^+(G).
\]
Now since $2k+1<n$, by induction $\mu^+(A_{2k+1}) \leq \mu^+(\tilde{G}) $ and thus $\mu^+(A_{2k+1}) \leq \mu^+(G)$.  Since $n$ is odd then by the Parity Principle we have $\mu^+(A_n) \leq \mu^+(A_{2k+1})$ and thus $\mu^+(A_n)\leq\mu^+(G)$ as desired.
\end{proof}
In the case that $n$ is odd, the critical threshold graphs not covered by Theorem~\ref{thm:odd-opt}(iii) have binary string $b=0^{s_1}1^{t_1}\cdots 0^{s_k}1^{t_k}$ with $s_1=1$ and $2k+1=n$, and thus only one of $s_2,\ldots,s_k,t_1,\ldots,t_k$ equals two and all others equal one.  There are only $2k-1=n-2$ such threshold graphs.

\section{Threshold Eigenvalue Estimates}
In this section, we exploit the subgraph structure of threshold graphs to give eigenvalue estimates for the non-trivial eigenvalues and the maximum/minimum eigenvalues.

Let $G$ be a threshold graph with binary string $b=0^{s_1}1^{t_1}\cdots 0^{s_k}1^{t_k}$.  Let $s=\min s_i$, let $t=\min t_i$, let $\sigma=\max s_i$, and let $\tau = \max t_i$.  Let $b' = 0^{s}1^{t}\cdots 0^{s}1^{t}$ and let $b'' = 0^{\sigma} 1^{\tau} \cdots 0^{\sigma} 1^{\tau}$ where in $b'$ the string $0^s1^t$ is repeated $k$ times and similarly for $b''$.  Let $n'=k(s+t)=|G'|$ and let $n''=k(\sigma+\tau)=|G''|$.  Then $G'$ is an induced subgraph of $G$ and $G$ is an induced subgraph of  $G''$.  
\begin{example}
If $b=0^31^20^41^60^51^3$, then $b'=0^31^20^31^20^31^2$ and $b''=0^51^60^51^60^51^6$.  The graphs $G$, $G'$, and $G''$ are illustrated in Figure~\ref{fig:last-example}.
\begin{figure}[t]
\centering
\includegraphics[height=90mm,keepaspectratio,trim=1.5in 3in 1in 2in,clip,keepaspectratio]{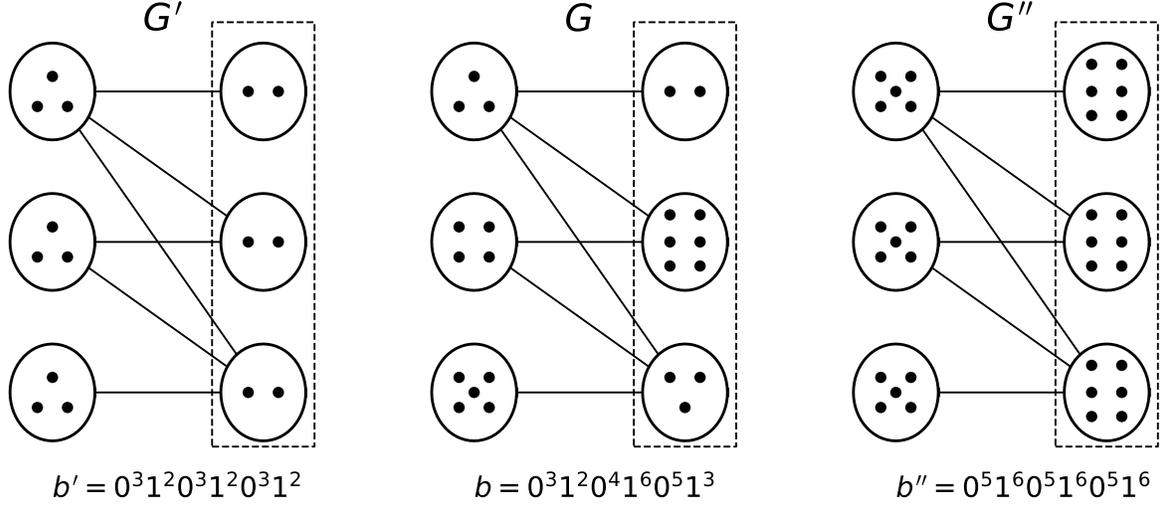}
\caption{Graphs from Example~\ref{exm:last-example}.}\label{fig:last-example}
\end{figure}  
\end{example}

A direct application of the eigenvalue interlacing theorem proves the following.
\begin{theorem}
With the above notation it holds that
\[
\lambda_i(G'') \leq \lambda_i(G) \leq \lambda_i(G'),\quad i=1,\ldots, k
\]
and
\[
\lambda_{n'-j}(G') \leq \lambda_{n-j}(G) \leq \lambda_{n''-j}(G''), \quad j=0,\ldots,k-1.
\]
\end{theorem}
We note that threshold graphs with binary sequence of the form $b=0^{s}1^{t}\cdots 0^{s}1^{t}$, that is, the independent sets $U_i$ all have the same number of vertices and similarly for the cliques $V_i$, share similar spectral properties as anti-regular graphs.

We end the paper with an estimate of the largest $\lambda_{\max}(G)$ and smallest $\lambda_{\min}(G)$ eigenvalues of a general threshold graph $G$.
\begin{theorem}\label{thm:max-min}
Let $G$ be a threshold graph with binary string $b=0^{s_1}1^{t_1}\ldots 0^{s_k}1^{t_k}$ and let $\sigma_i=\sum_{j=1}^i s_j$ and let $\tau_i=\sum_{j=i}^k t_j$ for each $i\in \{1,2,\ldots,k\}$.  Then
\[
\max_{1\leq i\leq k} \left\{ \frac{ (\tau_i - 1) + \sqrt{(\tau_i-1)^2 + 4\tau_i \sigma_i} }{2} \right\} \leq \lambda_{\textup{max}}(G)
\]
and
\[
\lambda_{\textup{min}}(G) \leq \min_{1\leq i\leq k} \left\{ \frac{ (\tau_i - 1) - \sqrt{(\tau_i-1)^2 + 4\tau_i \sigma_i} }{2} \right\}.
\]
\end{theorem}
\begin{proof}
By construction, the threshold graph with binary string $0^{\sigma_i}1^{\tau_i}$, which we denote by $G_i$, is an induced subgraph of $G$.  Using the quotient graph associated to the degree partition of $G_i$ (see for instance \cite{DJ-VT-FT:13}), the only non-trivial eigenvalues of $G_i$ are
\begin{align*}
\lambda_{\min}(G_i) &= \frac{ (\tau_i - 1) - \sqrt{(\tau_i-1)^2 + 4\tau_i \sigma_i} }{2} < -1\\[2ex]
\intertext{and}
\lambda_{\max}(G_i) & = \frac{ (\tau_i - 1) + \sqrt{(\tau_i-1)^2 + 4\tau_i \sigma_i} }{2} > 0.
\end{align*}
The claim now holds by eigenvalue interlacing.
\end{proof}
We note that it was proved in \cite{DJ-VT-FT:13} that among all threshold graphs on $n$ vertices, the minimum eigenvalue is minimized by the graph $0^s1^t$ where $t=\lfloor\frac{n}{3}\rfloor$ and $s=n-t$.

\begin{example}\label{exm:last-example}
Let $G$ be the threshold graph with binary string $b=0^{s_1}1^{t_1}\cdots 0^{s_5}1^{t_5}$ where $s=(2,2,3,6,3)$ and $t=(6,9,1,2,4)$.  Below we tabulate $\lambda_{\min}(G_i)$ and $\lambda_{\max}(G_i)$ up to five decimal places for $i=1,2,\ldots,5$ and indicate the minimum and maximum.
\begin{table}[h]
\centering
\begin{tabular}{ccc}\hline
$\lambda_{\min}(G_i)$ && $\lambda_{\max}(G_i)$\\ \hline
$-1.91974$ && $\mathbf{22.91974}$\\
$-3.46586$ && $18.46586$\\
$-4.61577$ && $10.61577$\\
$\mathbf{-6.67878}$ && $11.67878$\\
$-6.63941$ && $9.63941$\\ \hline
\end{tabular}
\end{table}

\noindent Using numerical software we found that $\lambda_{\min}(G) \approx -7.95182$ and $\lambda_{\max}(G) \approx 24.59001$.
\end{example}

\section{Conclusion}
In this paper we demonstrated the important role played by the anti-regular graph in the spectral analysis of threshold graphs.  The widely studied class of \textit{cographs} contain the threshold graphs as a special case.  It would be interesting to know if within the class of cographs there is a distinguished graph (or more) that can be used to analyze the spectral properties of cographs.


\section{Acknowledgements}
The authors acknowledge the support of the National Science Foundation under Grant No. ECCS-1700578.


\end{document}